\documentclass[11pt]{article}
\usepackage{amssymb,amsthm,amsmath,tikz,pgf,mathtools,subfigure}
\usepackage[lined,linesnumbered,ruled]{algorithm2e}
\usepackage{epsfig,amsmath, amsfonts,amstext, amsthm, latexsym, graphicx}
\usepackage{epstopdf}
\usepackage{graphicx}

\usetikzlibrary{graphs}

\newtheorem{defn}{Definition} 
\newtheorem{thm}{Theorem}
\newtheorem{lem}{Lemma} 

\DeclarePairedDelimiter\ceil{\lceil}{\rceil}

\begin{document}

\title{Cops and Robbers on Toroidal Chess Graphs}

\author{Allyson Hahn\\
	North Central College\\
	amhahn$@$noctrl.edu\\
	\and 
    Neil R. Nicholson\\
	North Central College\\
	nrnicholson$@$noctrl.edu\\
	}

\date{}

\maketitle

\begin{abstract} 
We investigate multiple variants of the game Cops and Robbers.  Playing it on an $n \times n$ toroidal chess graph, the game is varied by defining moves for cops and robbers differently, always mimicking moves of certain chess pieces.  In these cases, the cop number is completely determined.
\end{abstract}


\section{Introduction}
\label{intro}

Played on graphs, Cops and Robbers is a pursuit-evasion game between a set of cops and a set of robbers.  Briefly, the game is played as follows.  This description considers only a single robber, as in this paper we play the game with only one robber.

\begin{enumerate}
\item Given any mathematical graph $G$, $k$ cops choose up to $k$ vertices of $G$ as their starting positions (multiple cops may occupy a single vertex).
\item A single robber chooses a vertex as his starting position.
\item The cops take a turn: each cop can either stay on the vertex it currently occupies or move to an allowable\footnote{In the standard game of Cops and Robbers, a cop or robber moves by either staying on his current vertex or moving to occupy an adjacent vertex (that is, connected via an edge) to the one he currently occupies.  In Sec. \ref{defs}, we define moves differently.} vertex.
\item The robber takes a turn: the robber can either stay on the vertex he currently occupies or move to an allowable vertex.
\item The game proceeds with the cops and the robber alternating turns.
\item The cops win if, in a finite number of moves, at least one cop can occupy the same vertex as the robber.
\item The robber wins if he can guarantee to never have a cop occupy the same vertex as he occupies.
\end{enumerate}

Introduced in the early 1980s \cite{Winkler,Quillot}, it has blossomed into numerous variations and these twists continue to inspire interesting research.  The game can be varied by restricting the types of graphs it is played on \cite{Alon,Dawes,Frankl,Neufeld} and/or providing alternate definitions of how the cops or robbers move \cite{Bal,Sullivan}.  Regardless of the variation being played, certain questions are commonly asked of the game.  In particular, what is the fewest number of cops needed to guarantee their victory (called the \textit{cop number} of the graph), and if the cops can win, what is the fewest number of moves needed to win?  

These are just two of the questions asked about certain variations of the game; by no means are they the only questions of importance.  Entire texts have been written summarizing past research, applications, and open questions stemming from the original game \cite{Bonato}.  But our focus here is variants of the game played on $n \times n$ toroidal chess graphs (also called grid graphs).  Moves of the cops and robber will be defined to mirror those of various pieces in the game of chess.  While the cop number is determined in some of these scenarios in Sec. \ref{mainresults}, and those that are not investigated are described in Sec. \ref{futurequestions}.  We begin with the necessary terminology in Sec. \ref{defs}.

\section{Definitions and Preliminaries}
\label{defs}

Throughout the paper, we will assume $G$ is an $n\times n$ chess graph on the torus.  A \textit{chess graph} (also called a \textit{grid graph}) is an $n \times n$ array of vertices, with an edge between all horizontal ``neighboring" vertices as well as all vertical ``neighboring" vertices.  Placing such a graph on the torus equates to placing an additional edge from the first to the last vertex in every row, as well as the first to the last vertex in every column.  We do this because a torus can be formed by identifying the edges of a square, as in Fig. \ref{fig1}.  

We will visualize our chess graphs on the torus as simply a planar $n \times n$ array of vertices, as in Fig. \ref{fig2}. This allows us to talk about ``positions" on the chessboard, whereas visualizing the chessboard on the torus does not allow us to do so.  On the plane, referring to the ``top row of vertices" or ``moving one vertex right" makes sense; it does not when viewed explicitly on the three-dimensional torus.

\begin{figure} 
\begin{center}
\begin{tikzpicture}[thick,scale=.8]
  \node (0) at (-1,0) {};
  \node (1) at (1,0)  { };
  \draw[ultra thick, ->] (0) edge (1);
  
  \draw (-7,-2)rectangle(-3,2);
  
  \node (2) at (-7,-2) {};
  \node (3) at (-5,-2) {};
  \node (4) at (-3,-2) {};
  \node (5) at (-7,0) {};
  \node (6) at (-3,0) {};
  \node (7) at (-7,2) {};
  \node (8) at (-5,2) {};

  \draw[->] (2) edge (3);
  \draw[->] (7) edge (8);
  \draw[->>] (2) edge (5);
  \draw[->>] (4) edge (6);
  
\draw (7,0) ellipse (4 and 2);
\draw[rounded corners=28pt] (4.9,.2)--(7,-.3)--(9.1,.2);
\draw[rounded corners=24pt] (5.8,0)--(7,.3)--(8.2,0);

\draw[densely dashed] (7,-2) arc (270:90:.5 and .916);
\draw (7,-2) arc (-90:90:.5 and .916);

\node (9) at (7.46,-.61) {};
\node (10) at (7.46,-.6) {};
\draw[->>] (9) edge (10);

\draw (11,0) arc (0:-180:4 and 1.2);
\draw[densely dashed] (11,0) arc (0:180:4 and 1.2);

\node (11) at (9,-1.07) {};
\node (12) at (9.01,-1.07) {};
\draw[->] (11) edge (12);

\end{tikzpicture}
\end{center}
\caption{Identifying edges of a square to create a torus} \label{fig1}
\end{figure}
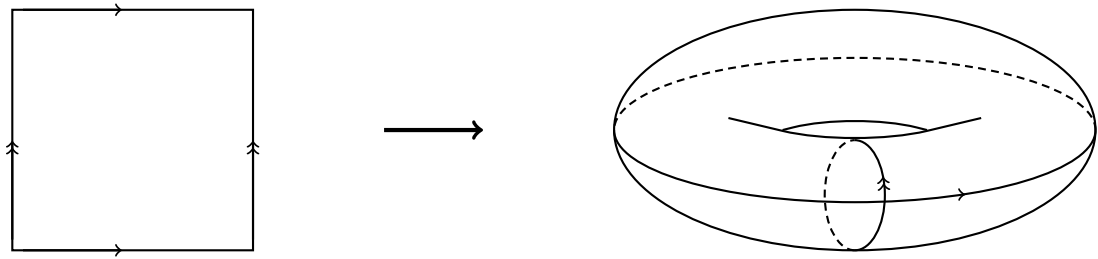

\begin{figure} 
\begin{center}
\begin{subfigure}[A $4 \times 4$ chess graph in the plane]{
\label{fig2planar}

\begin{tikzpicture}[scale=.6]

 \draw[dashed] (-4,-1)--(4,-1)--(4,7)--(-4,7)--(-4,-1);
  \path (-5,-1) rectangle(5,7);
 
 \foreach \x in {-3,-1,1,3}
   \foreach \y in {0,2,4,6}{
   \node at (\x,\y)[circle,fill,inner sep=2pt]{};
   }; 

 \draw (-3,0)--(3,0);
 \draw (-3,2)--(3,2);
 \draw (-3,4)--(3,4);
 \draw (-3,6)--(3,6);
 
 \draw (-3,0)--(-3,6);
 \draw (-1,0)--(-1,6);
 \draw (1,0)--(1,6);
 \draw (3,0)--(3,6);
\end{tikzpicture}
}
\end{subfigure}
\begin{subfigure}[A $4 \times 4$ chess graph on the torus]{
\label{fig2torus}

\begin{tikzpicture}[scale=.6]
  \draw[dashed] (-4,-1)--(4,-1)--(4,7)--(-4,7)--(-4,-1);
  \node (1) at (-4,3) {};
  \node (2) at (-4,3.01) {};
  \draw[->>] (1) edge (2);
\node (3) at (4,3) {};
  \node (4) at (4,3.01) {};
  \draw[->>] (3) edge (4);
\node (5) at (-.01,-1) {};
  \node (6) at (0,-1) {};
  \draw[->] (5) edge (6);
  \node (7) at (-.01,7) {};
  \node (8) at (0,7) {};
  \draw[->] (7) edge (8);
  
  \path (-5,-1) rectangle(5,7);
  
   \foreach \x in {-3,-1,1,3}
   \foreach \y in {0,2,4,6}{
   \node at (\x,\y)[circle,fill,inner sep=2pt]{};
   }; 
   
 \draw (-4,0)--(4,0);
 \draw (-4,2)--(4,2);
 \draw (-4,4)--(4,4);
 \draw (-4,6)--(4,6);
 
 \draw (-3,-1)--(-3,7);
 \draw (-1,-1)--(-1,7);
 \draw (1,-1)--(1,7);
 \draw (3,-1)--(3,7);
\end{tikzpicture}
}
\end{subfigure}

\end{center}
\caption{Chess graph in the plane and on the torus}
 \label{fig2} 
\end{figure}

The basic question in any specific game of Cops and Robbers is whether the cops or the robber will win.  More generally, one can ask the question, ``What is the fewest number of cops needed to guarantee victory?"  

\begin{defn}
The \textbf{cop number} of a graph $G$, denoted $c(G)$, is the minimum number of cops required to guarantee the existence of a winning strategy for the cops, regardless of the robber's initial position.
\end{defn}

Typically, a move is a choice by a cop or robber to remain on the vertex they currently occupy or to move to an adjacent vertex on the graph.  If the cops and robber all move in the same fashion, the graph can be defined in such a way to ``allow" the intended moves.  For example, if the cops and robbers are meant to mimic the moves of a rook in the game of chess, then placing an edge between every pair of vertices in the same row or column of the standard grid graph accomplishes this. However, if the cops and robber move differently, such as cops moving as rooks yet the robber moves as a pawn, then this is not possible. Defining a move for both the cop and the robber as simply moving to an adjacent vertex fails.  

To fix this problem, one can either weight edges of the graph (allowing cops and the robber to move across edges only of a certain weight), or, as we do in this paper, define a move differently than in most literature on the game of Cops and Robbers.  To that end, we have the following definitions. 

\begin{defn} \label{allowabledef}
The vertices that a cop or a robber could move to are called \textbf{allowable vertices}, and the allowable vertices for a cop's move are said to be \textbf{protected} by the cop.
\end{defn}

\begin{defn} \label{knightsdef}
Suppose $G$ is a chess graph (either on the plane or on the torus).  A cop (or robber) is \textbf{on foot} if the allowable vertices for a move are those adjacent to the vertex occupied by the cop (or robber). They are a \textbf{knight} if the allowable vertices for a move are those located $2$ columns and $1$ row, or, $1$ column and $2$ rows away from the vertex occupied by the cop (or robber).  That is, allowable vertices for a knight's move correspond to those of a knight in the game of chess.
\end{defn}

\begin{figure} 
\begin{center}
\begin{tikzpicture}[scale=.6]

 \foreach \x in {-4,-2,0,2,4}
   \foreach \y in {0,2,4,6,8}{
   \node at (\x,\y)[circle,fill,inner sep=2pt]{};
   }; 

  \node at (0,4)[circle,fill,inner sep=5pt]{};
  \node at (0,2.9)[]{Cop};
  
  \draw (-2,0) circle (9pt);
  \draw (-2,8) circle (9pt);
  \draw (-4,2) circle (9pt);
  \draw (-4,6) circle (9pt);
  \draw (2,0) circle (9pt);
  \draw (2,8) circle (9pt);
  \draw (4,2) circle (9pt);
  \draw (4,6) circle (9pt);
  
  \draw (-4.5,0)--(4.5,0);
  \draw (-4.5,2)--(4.5,2);
  \draw (-4.5,4)--(4.5,4);
  \draw (-4.5,6)--(4.5,6);
  \draw (-4.5,8)--(4.5,8);
  
  \draw (-4,-.5)--(-4,8.5);
  \draw (-2,-.5)--(-2,8.5);
  \draw (0,-.5)--(0,2.4);
  \draw (0,4)--(0,8.5);
  \draw (2,-.5)--(2,8.5);
  \draw (4,-.5)--(4,8.5);
\end{tikzpicture}
\end{center}
\caption{Vertices protected (circled) by a cop as a knight} \label{fig3}
\end{figure}
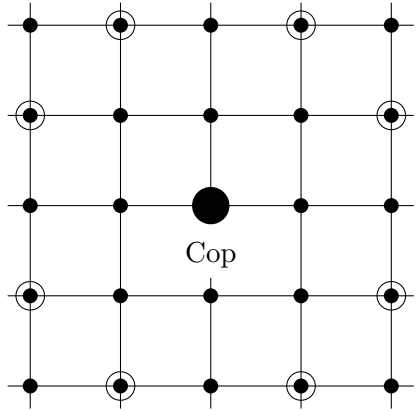

\section{Main Results}
\label{mainresults}

The standard question addressed in the game of Cops and Robbers is when both all of the cops and the robber are on foot.  We begin this section by investigating the situation where the cops are knights and the robber is on foot.  

\subsection{Cops as knights, robber on foot}
\label{copsknightrobberfoot}

In general, a large enough number of cops always has a winning strategy, regardless of how the cops move (a graph with $n$ vertices will have a cop number no greater than $n$).  Theorems \ref{smallchessknights} and \ref{arbitraryknights} determine, in the situation when cops are knights and the robber is on foot, an explicit cop number for every $n \times n$ ($n \geq 3$) chess graph on the torus.

\begin{thm} \label{smallchessknights}
Let $G$ be the $3 \times 3$ or $4 \times 4$ chess graph on the torus, with cops as knights and the robber on foot.  Then, $c(G) = 2$.
\end{thm}
\begin{proof}
In either case, a robber can indefinitely evade a single cop, since one cop cannot protect the $5$ allowable vertices for robber's move (see Fig. \ref{fig3}). Thus, we need only show that $2$ cops have a winning strategy.

Let $G$ be the $3 \times 3$ graph on the torus.  Place the two cops on the leftmost vertices of the first row.  Every vertex in the second and third row is protected, meaning the robber can only place himself on the remaining vertex of the first row.  Choose to have both cops move to the second row.  In doing so, the vertex immediately below the robber will be occupied by a cop.  Thus, the robber is unable to move to the second row (and consequently to the only unprotected vertex), meaning no matter where he moves (or stays), he will be on a vertex protected by a cop.  Hence, the cops have a winning strategy.

Suppose now that $G$ is the $4 \times 4$ graph on the torus.  Place the two cops on the leftmost vertices of the first row.  We will choose to move the two cops identically for any possible move.  In doing so, we can always consider the resulting chess graph to look like that in Fig. \ref{fig8} (since $G$ is on the torus). Note that there are two vertices (vertices $\star$ and $\star\star$ in Fig. \ref{fig8}) that if occupied by the robber would guarantee his capture within two moves by the cops, as all allowable vertices for the robber's moves are protected by the two cops.  We need only show that regardless of the robber's initial position, then, the cops are able to move so that the robber then occupies one of these guaranteed victory vertices.  The initial location of the robber (the non-protected vertices of Fig. \ref{fig8}) and such a move by the cops are listed below.

\begin{center}
\begin{tabular}{c|c|c}
Robber's initial position & Cop's move & Robber's resulting position\\ \hline
Vertex $1$ & up $2$, left $1$ & Vertex $ \star\star  $\\
Vertex $2$ & up $2$, right $1$ & Vertex $ \star$\\
Vertex $3$ & up $1$, left $2$ & Vertex $\star $\\
Vertex $4$ & up $1$, left $2$ & Vertex $\star\star $\\
Vertex $5$ & down $1$, left $2$ & Vertex $\star $ \\
Vertex $6$ & down $1$, left $2$ & Vertex $\star\star $\\
\end{tabular}
\end{center}
\end{proof}

\begin{figure} 
\begin{center}
\begin{tikzpicture}[scale=.6]

 \foreach \x in {-3,-1,1,3}
   \foreach \y in {0,2,4,6}{
   \node at (\x,\y)[circle,fill,inner sep=2pt]{};
   }; 
  
  \node at (-3,6)[circle,fill,inner sep=5pt]{};
  \node at (-3,5)[]{Cop};

  \node at (-1,6)[circle,fill,inner sep=5pt]{};
  \node at (-1,5)[]{Cop};
  
  \draw (-3.5,0)--(3.5,0);
  \draw (-3.5,2)--(3.5,2);
  \draw (-3.5,4)--(3.5,4);
  \draw (-3.5,6)--(3.5,6);
  
  \draw (1,4) circle (9pt);
  \draw (3,4) circle (9pt);
  \draw (-3,2) circle (9pt);
  \draw (-1,2) circle (9pt);
  \draw (1,2) circle (9pt);
  \draw (3,2) circle (9pt);
  \draw (1,0) circle (9pt);
  \draw (3,0) circle (9pt);  
  
  \node at (1,1)[]{$\star$};
  \node at (3,1)[]{$\star \star$};
  \node at (1,5)[]{$1$};
  \node at (3,5)[]{$2$};
  \node at (-3,3)[]{$3$};
  \node at (-1,3)[]{$4$};
  \node at (-3,-1)[]{$5$};
  \node at (-1,-1)[]{$6$};
  
  \draw (-3,-.3)--(-3,2.5);
  \draw (-3,3.5)--(-3,4.5);
  \draw (-3,5.5)--(-3,6.5);
  
  \draw (-1,-.3)--(-1,2.5);
  \draw (-1,3.5)--(-1,4.5);
  \draw (-1,5.5)--(-1,6.5);
  
  \draw (1,-.5)--(1,.5);
  \draw (1,1.5)--(1,3.5);
  \draw (1,5.5)--(1,6.5);
  
  \draw (3,-.5)--(3,.5);
  \draw (3,1.5)--(3,3.5);
  \draw (3,5.5)--(3,6.5);

\end{tikzpicture}
\end{center}
\caption{A cop's winning strategy on the $4 \times 4$ chess graph on the torus} \label{fig8}
\end{figure}
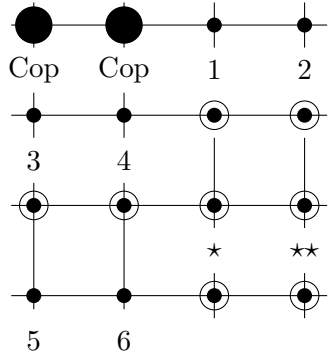

The previous theorem proves that just two knights can capture a robber on foot on the $3 \times 3$ or $4 \times 4$ chess graph on the torus.  The following lemma proves that a robber on foot can always evade the capture of two knights on a larger $n \times n$ chess graph on the torus.  The theorem that follows, however, proves that regardless how large $n$ is, just three knights suffice in capturing the robber.

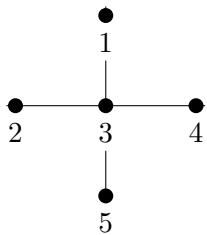
\begin{figure} 
\begin{center}
\begin{tikzpicture}[scale=.6]

   \node at (-2,2)[circle,fill,inner sep=2pt]{};
   \node at (0,0)[circle,fill,inner sep=2pt]{};
   \node at (0,2)[circle,fill,inner sep=2pt]{};
   \node at (0,4)[circle,fill,inner sep=2pt]{};
   \node at (2,2)[circle,fill,inner sep=2pt]{};
   
   \node at (0,3.4)[]{$1$};
   \node at (-2,1.4)[]{$2$};
   \node at (0,1.4)[]{$3$};
   \node at (2,1.4)[]{$4$};
   \node at (0,-.6)[]{$5$};

  \draw (-2.2,2)--(2.2,2);
  \draw (0,4)--(0,4.2);
  \draw (0,2)--(0,3);
  \draw (0,0)--(0,1);
\end{tikzpicture}
\end{center}
\caption{Unprotectable by two cops} \label{fig4}
\end{figure}

\begin{lem} \label{5vertslemma}
If $G$ is the $n \times n$ chess graph on the torus, $n \geq 5$, then two cops as knights cannot protect the five vertices as in Fig. \ref{fig4}.
\end{lem}
\begin{proof}
Suppose two cops as knights are placed on an $n \times n$, $n \geq 5$, chess graph on the torus.  First note that any cop located on a vertex of Fig. \ref{fig4} could not protect any of the other vertices (in reference to Fig. \ref{fig4}, as will be all vertices mentioned in this proof).  Moreover, Fig. \ref{fig3} shows that a single cop cannot protect all of vertices $2$, $3$, and $4$ .  We have, then, two cases to consider: of these three vertices, one cop protects two adjacent ones or neither cop protects adjacent ones.

Without loss of generality, consider the first case of cop $A$ protecting vertices $2$ and $3$.  Figure \ref{fig3} shows that cop $A$ cannot protect vertices $1$ or $5$, meaning cop $B$ must protect all of $1$, $4$, and $5$. This is impossible, as Fig. \ref{fig3} exhibits.

Then consider the case of cop $A$ protecting vertices $2$ and $4$.  Again, Fig. \ref{fig3} shows that $A$ cannot protect any of vertices $1$, $3$ or $5$ (since $n \geq 5$).  Yet it is impossible for cop $B$ to protect these three adjacent vertices as well, proving the desired result.
\end{proof}

Armed with this lemma, we are ready to proceed in determining the cop number for the general $n \times n$, $n \geq 5$, chess graph on the torus.

\begin{thm} \label{arbitraryknights}
If $G$ is the $n \times n$ chess graph on the torus, $n \geq 5$, with all cops being knights and the robber being on foot, then $c(G) = 3$.
\end{thm}
\begin{proof}
Let $G$ be the $n \times n$ ($n \geq 5$) chess graph on the torus, and assume all cops are knights and the robber is on foot.  We prove $c(G) = 3$ by showing that the robber can indefinitely evade $2$ cops but that $3$ cops have a winning strategy.\\
 
\noindent \textit{Claim 1}: $c(G) > 2$\\

On any given move, a robber has five potential options: stay on the vertex he is located on or move to one of the four adjacent vertices to his current location.  These locations correspond to the five vertices of Fig. \ref{fig4} and Lemma \ref{5vertslemma}.  Because one of these vertices will always be unprotected when there are two cops, the robber can guarantee he is never caught (staying at his current location, vertex $3$ in Fig. \ref{fig4}, if that vertex is unprotected, or moving to whichever of vertices $1$, $2$, $4$, or $5$ is unprotected.).  Thus, $c(G) > 2$.\\

\noindent \textit{Claim 2}: Three cops placed as in Fig. \ref{fig5} can, in a finite number of moves, move so that the robber is on a vertex within the rectangular array.\\

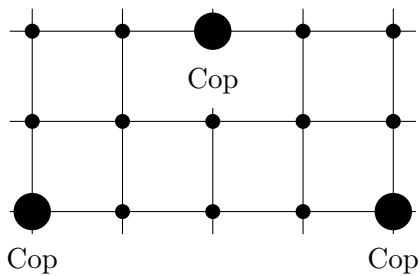
\begin{figure} 
\begin{center}
\begin{tikzpicture}[scale=.6]

 \foreach \x in {-4,-2,0,2,4}
   \foreach \y in {0,2,4}{
   \node at (\x,\y)[circle,fill,inner sep=2pt]{};
   }; 

  \node at (-4,0)[circle,fill,inner sep=5pt]{};
  \node at (-4,-1.1)[]{Cop};
    
  \node at (0,4)[circle,fill,inner sep=5pt]{};
  \node at (0,2.9)[]{Cop};
  
  \node at (4,0)[circle,fill,inner sep=5pt]{};
  \node at (4,-1.1)[]{Cop};

  \draw (-4.5,0)--(4.5,0);
  \draw (-4.5,2)--(4.5,2);
  \draw (-4.5,4)--(4.5,4);
  
  \draw (-4,-.5)--(-4,4.5);
  \draw (-2,-.5)--(-2,4.5);
  \draw (0,-.5)--(0,2.3);
  \draw (0,3.6)--(0,4.5);
  \draw (2,-.5)--(2,4.5);
  \draw (4,-.5)--(4,4.5);
  
\end{tikzpicture}
\end{center}
\caption{Positioning cops to ``chase" the robber} \label{fig5}
\end{figure}

Recall that the distance between two vertices is the length of the shortest path between them.  Then, consider subsequent moves by a cop and the robber. Suppose they are on vertices whose distance between them is $n$ units ($n$ arbitrarily large).  A single move by the robber can increase this distance by at most $1$ unit, whereas a move by the cop (a knight) can decrease this distance by $3$ units.  Hence, in subsequent moves, when the cop and robbers are sufficiently far enough apart, the distance between the two can be guaranteed to decrease.  Moreover, the cop can, in a finite number of moves, move to be no more than a distance of $2$ units from the robber.

If three cops are placed on the graph as in Fig. \ref{fig5}, choose to always move the three cops identically, preserving the arrangement of the cops.  By the preceding paragraph, in a finite number of moves, the arrangement of cops can move so that the middle vertex of this array (second row, third column) is no more than $2$ units away from the robber.  This means that the robber is located either within the array, as desired, or located directly on the vertex above or the vertex three below the cop in the top row, just outside the array.  Regardless of his next move, the cops can choose to move two units (up if the robber is located above the array; down if the robber is below the array) and one unit right.  Upon doing so, the robber will lie on a vertex in the array, as desired.

Thus, by Claim $2$, it is enough to show that the cops can catch the robber once he lies within the array of Fig. \ref{fig5}.  To do this, we will rely on Claim $3$.\\

\noindent \textit{Claim 3}: If the robber is ever in the position of Fig. \ref{fig6} (or some symmetrical variant of them), he will be caught.\\

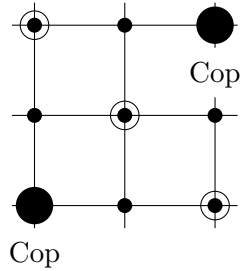
\begin{figure} 
\begin{center}
\begin{tikzpicture}[scale=.6]

 \foreach \x in {-2,0,2}
   \foreach \y in {0,2,4}{
   \node at (\x,\y)[circle,fill,inner sep=2pt]{};
   }; 

  \node at (-2,0)[circle,fill,inner sep=5pt]{};
  \node at (-2,-1.1)[]{Cop};
    
  \node at (2,4)[circle,fill,inner sep=5pt]{};
  \node at (2,2.9)[]{Cop};

  \draw (-2.5,0)--(2.5,0);
  \draw (-2.5,2)--(2.5,2);
  \draw (-2.5,4)--(2.5,4);
  
  \draw (0,-.5)--(0,4.5);
  \draw (-2,-.5)--(-2,4.5);
  \draw (2,-.5)--(2,2.4);
  \draw (2,3.5)--(2,4.5);

  \draw (-2,4) circle (9pt);
  \draw (0,2) circle (9pt);
  \draw (2,0) circle (9pt);
  
\end{tikzpicture}
\end{center}
\caption{Locations guaranteeing capture (circled)} \label{fig6}
\end{figure}

Notice that the two cops in Fig. \ref{fig6} protect all vertices adjacent to the robber's position (but not the vertex the robber lies on).  Thus, the robber cannot move without being caught.  At this point, the third cop can repeatedly move until he protects the robber's vertex (which is possible since every space on a chessboard can be reached by a knight). 

Let us consider then the possible locations of the robber once he is located within the array of Fig. \ref{fig5}.  Due to the symmetry of the array, it is enough to consider the labeled vertices  of Fig. \ref{fig7}.

\begin{figure} 
\begin{center}
\begin{tikzpicture}[scale=.6]

 \foreach \x in {-4,-2,0,2,4}
   \foreach \y in {0,2,4}{
   \node at (\x,\y)[circle,fill,inner sep=2pt]{};
   }; 

  \node at (-4,0)[circle,fill,inner sep=5pt]{};
  \node at (-4,-1.1)[]{$A$};
    
  \node at (0,4)[circle,fill,inner sep=5pt]{};
  \node at (0,3)[]{$B$};
  
  \node at (4,0)[circle,fill,inner sep=5pt]{};
  \node at (4,-1.1)[]{$C$};

  \draw (-4.5,0)--(4.5,0);
  \draw (-4.5,2)--(4.5,2);
  \draw (-4.5,4)--(4.5,4);
  
  \draw (-4,-.5)--(-4,.55);
  \draw (-4,1.5)--(-4,2);
  \draw (-4,2)--(-4,2.55);
  \draw (-4,3.5)--(-4,4.5);
  
  \draw (-2,-.5)--(-2,.55);
  \draw (-2,1.5)--(-2,2);
  \draw (-2,2)--(-2,2.55);
  \draw (-2,3.5)--(-2,4.5);
  
  \draw (0,-.5)--(0,.55);
  \draw (0,1.5)--(0,2);
  \draw (0,2)--(0,2.55);
  \draw (0,3.5)--(0,4.5);
  
  \draw (0,3.6)--(0,4.5);
  
  \draw (2,-.5)--(2,4.5);
  \draw (4,-.5)--(4,4.5);

  \draw (-4,4) circle (9pt);
  \node at (-4,3)[]{$1$};
  \draw (-2,4) circle (9pt);
  \node at (-2,3)[]{$2$};
  \draw (-4,2) circle (9pt);
  \node at (-4,1)[]{$3$};
  \draw (-2,2) circle (9pt);
  \node at (-2,1)[]{$4$};
  \draw (0,2) circle (9pt);
  \node at (0,1)[]{$5$};
  \draw (-2,0) circle (9pt);
  \node at (-2,-1.1)[]{$6$};
  \draw (0,0) circle (9pt);
  \node at (0,-1.1)[]{$7$};  
\end{tikzpicture}
\end{center}
\caption{Cops $A$, $B$, $C$ and robber locations (circled) to consider} \label{fig7}
\end{figure}
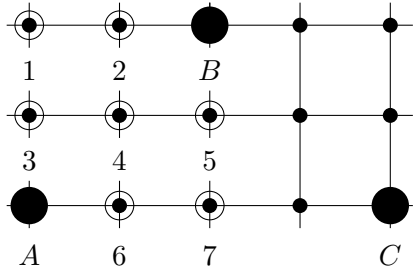

Vertices $1$, $4$ and $7$, by Claim $3$, guarantee the robber's capture.  Consider then the remaining four possible locations of the robber.  

Vertex $5$: The robber must move, as he is on a vertex protected by cop $A$.  However, moving to any vertex adjacent to his position will guarantee his capture.  He cannot move vertically, as moving upwards is not an option (the vertex is occupied by cop $B$), and Claim $5$ guarantees his capture if he moves down.  A horizontal move puts him on vertex $4$ (or symmetrical opposite vertex), both of which guarantee his capture by the comments in the preceding paragraph.  

Vertex $6$: Since vertex $6$ is protected by cop $B$, the robber must move.  His only option is to move to the vertex directly below vertex $6$, similar to the single move the robber could make when sitting on vertex $5$.  Upon his moving, choose to move all three cops simultaneously down two vertices and left one vertex.  The result is that the robber is now occupying vertex $5$ of Fig. \ref{fig7}.  As we see above, the cops have guaranteed victory in this scenario.

Vertex $3$:  If the robber is on vertex $3$, as in previous cases, he only has one option for escaping capture. In this scenario he must move left one vertex.  As in the previous case, have all cops move identically, though in this case they move left two vertices and up one vertex.  The result is that the robber is now located on vertex $6$ of Fig. \ref{fig7}, guaranteeing his capture.

Vertex $2$: As in the other cases, the robber is forced to move to an adjacent vertex; he must move up one vertex.  Move all cops up two vertices and left one vertex. The robber now occupies vertex $6$ of Fig. \ref{fig7}, once again guaranteeing his capture.
\end{proof}

\subsection{Cops with a chief, speedy robber}
\label{}

The variation of the game in the previous subsection considered ``all cops to be equal," as in every cop moved identically.  Our attention shifts now to a variation of the game where all cops \textit{do not} all move identically.  In particular, one cop, whom we refer to as the ``chief," moves as a queen in the game of chess.  We will assume from this point forward that the graph $G$ is an arbitrarily large chess graph.

\begin{defn} \label{chiefdef}
 A cop is a \textbf{chief} if the allowable vertices for a move are those located in the same row, column, or diagonal as the vertex occupied by the cop.
\end{defn}

Determining the cop number, then, of a graph with one chief and the remaining cops all being on foot equates to simply asking, ``How many additional cops, all on foot, must  exist to capture the robber?"  Lemma \ref{robberonfoot} shows that if the robber is on foot, then the chief needs no additional help.\\

\begin{lem} \label{robberonfoot}
If the robber is on foot, then a single cop, as a chief, has a winning strategy.  That is, $c(G) = 1$.
\end{lem}
\begin{proof}
Place the robber (on foot) and the cop (a chief) on $G$.  Move the cop to the column in which the robber lies.  The robber, wanting to avoid capture, must move one vertex horizontally.  Next, move the cop to the vertex previously occupied by the robber, adjacent to the robber's current position.  Note that the cop protects all vertices the robber could subsequently move to, proving that the cop will capture the robber. 
\end{proof}

Let us expand this scenario (one chief, all other cops on foot) by considering cases where the robber can move differently.  We focus on what we call speedy robbers.

\begin{defn} \label{speedyrobberdef}
A robber (or cop) is said to be \textbf{$m$-speedy} ($m < n$) if the allowable vertices for a move are those in the same row or column located up to a distance of $m$ from the vertex occupied by the robber. 
\end{defn}

Note that a $1$-speedy robber is the same as a robber on foot.  Before investigating the general case, the following lemma allows a chief to ``trap" a robber into one particular row.\\  

\begin{lem} \label{queenpositionlemma}
A cop that is a chief can guarantee that that an $m$-speedy robber will never leave the row it originally occupies.
\end{lem}
\begin{proof}
Place the $m$-speedy robber and a cop that is a chief on $G$.  Move the cop into the column occupied by the robber.  Because every vertex in this column is protected by the cop, to avoid capture, the robber only has the option to move to another vertex in the row it occupies.  Have the cop mirror the move of the robber; that is, if the robber moves $k$ units horizontally, move the cop $k$ units horizontally.  The result will be that the cop will lie once again in the same column as the robber.  Continuing this, the robber will only be able to move within the row it originally occupied.
\end{proof}

We move then to addressing the question of determining the cop number if the robber is $m$-speedy ($m > 1$), one cop is a chief, and all remaining cops are on foot.  We can focus solely on determining the number of cops on foot both necessary and sufficient for capturing the robber.  \\

\begin{lem} \label{queenpositionlemma2}
To determine $c(G)$ when the robber is $m$-speedy and cops consist of a chief and all others on foot, it suffices to consider the situation where the cop that is a chief is consistently protecting the entire column (or row) that the robber occupies. 
\end{lem}
\begin{proof}
The cop number of any graph corresponds to the minimal number of cops needed to, after some finite number of moves, protect all vertices that the robber could move to.  If the robber is $m$-speedy, then cops must protect $4m+1$ vertices.  A cop on foot can protect at most $3$ of those vertices, while the chief can protect a variable number of vertices, depending on her position relative to the robber. Thus, maximizing the number of vertices protected by the chief minimizes the minimum number of cops on foot required to protect the remaining vertices.

If the chief is located in the same column (or row) as the robber, then only the other $2m$ vertices in the robber's row (or column) that he could move to need to be protected.  Let us consider the other possible locations of the chief, first assuming the chief is located both within $m$ rows and $m$ columns of the robber's position.  There are three situations to consider.

If the chief and the robber are located on the same diagonal, then the chief protects $3$ vertices that the robber could potentially move to (the robber's current position plus one in the column the chief occupies and one in the row the chief occupies).  Otherwise, the chief protects either $4$ or $6$ of the vertices the robber could move to.  When she is located greater than $8$ units away from the robber's position, she protects $4$ of these vertices; when the distance is $8$ or fewer, she protects $6$.  See Fig. \ref{queenprotecting}.

In these three cases (the chief protecting $3$, $4$, or $6$ of the necessary vertices), there remain $4m-2$, $4m-3$, or $4m-5$ remaining vertices for the cops on foot to protect, respectively.  It is only the case that this is less than $2m$ (fewer than when the chief protects an entire row or column) when $m < 1$ (not possible), $m < \frac{3}{2}$ (corresponds to the robber being on foot), or $m < \frac{5}{2}$, respectively.  Thus, we need only consider the last situation, when $m=2$.  But, when $m = 2$, vertices located within $2$ units of the robber's position must be in the same row, column, or diagonal as the robber.  Consequently, there are no vertices that the chief could occupy and protect $6$ of the vertices the robber could potentially move to.  In all cases, then, the chief protects a maximal number of the desired vertices when she is located in the same row or column as the robber, as was claimed.\\
\end{proof}

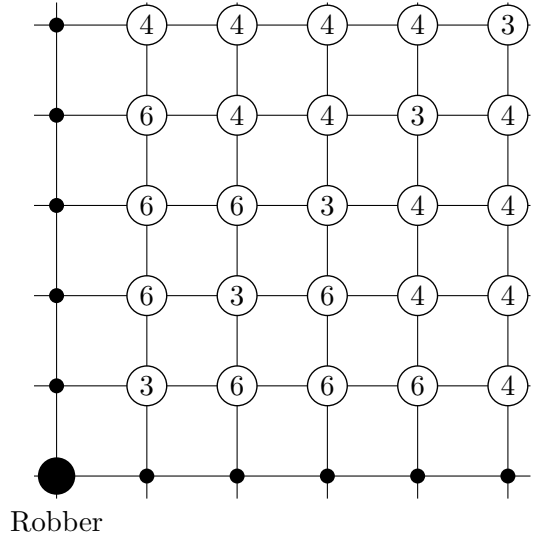
\begin{figure} 
\begin{center}
\begin{tikzpicture}[scale=.6]

 \foreach \x in {-4,-2,0,2,4,6}
   \foreach \y in {0,2,4,6,8,10}{
   \node at (\x,\y)[circle,fill,inner sep=2pt]{};
   }; 

  \node at (-4,0)[circle,fill,inner sep=5pt]{};
  \node at (-4,-1)[]{Robber};
  
  
  \draw (-4.5,0)--(6.5,0);
  \draw (-4.5,2)--(6.5,2);
  \draw (-4.5,4)--(6.5,4);
  \draw (-4.5,6)--(6.5,6);
  \draw (-4.5,8)--(6.5,8);
  \draw (-4.5,10)--(6.5,10);
  
  \draw (-4,-.5)--(-4,10.5);
  \draw (-2,-.5)--(-2,10.5);
  \draw (0,-.5)--(0,10.5);
  \draw (2,-.5)--(2,10.5);
  \draw (4,-.5)--(4,10.5);
  \draw (6,-.5)--(6,10.5);
  
  \node at (-2,2)[circle,fill,inner sep=5.5pt]{};
  \node at (-2,2)[circle,fill=white,inner sep=2pt]{$3$};
  \node at (0,4)[circle,fill,inner sep=5.5pt]{};
  \node at (0,4)[circle,fill=white,inner sep=2pt]{$3$};
  \node at (2,6)[circle,fill,inner sep=5.5pt]{};
  \node at (2,6)[circle,fill=white,inner sep=2pt]{$3$};
  \node at (4,8)[circle,fill,inner sep=5.5pt]{};
  \node at (4,8)[circle,fill=white,inner sep=2pt]{$3$};
  \node at (6,10)[circle,fill,inner sep=5.5pt]{};
  \node at (6,10)[circle,fill=white,inner sep=2pt]{$3$};

  \node at (0,2)[circle,fill,inner sep=5.5pt]{};
  \node at (0,2)[circle,fill=white,inner sep=2pt]{$6$};
  \node at (2,2)[circle,fill,inner sep=5.5pt]{};
  \node at (2,2)[circle,fill=white,inner sep=2pt]{$6$};
  \node at (4,2)[circle,fill,inner sep=5.5pt]{};
  \node at (4,2)[circle,fill=white,inner sep=2pt]{$6$};
  \node at (2,4)[circle,fill,inner sep=5.5pt]{};
  \node at (2,4)[circle,fill=white,inner sep=2pt]{$6$};
  \node at (-2,4)[circle,fill,inner sep=5.5pt]{};
  \node at (-2,4)[circle,fill=white,inner sep=2pt]{$6$};
  \node at (-2,6)[circle,fill,inner sep=5.5pt]{};
  \node at (-2,6)[circle,fill=white,inner sep=2pt]{$6$};
  \node at (-2,8)[circle,fill,inner sep=5.5pt]{};
  \node at (-2,8)[circle,fill=white,inner sep=2pt]{$6$};
  \node at (0,6)[circle,fill,inner sep=5.5pt]{};
  \node at (0,6)[circle,fill=white,inner sep=2pt]{$6$};
  
  \node at (4,4)[circle,fill,inner sep=5.5pt]{};
  \node at (4,4)[circle,fill=white,inner sep=2pt]{$4$};
  \node at (4,6)[circle,fill,inner sep=5.5pt]{};
  \node at (4,6)[circle,fill=white,inner sep=2pt]{$4$};
  \node at (6,2)[circle,fill,inner sep=5.5pt]{};
  \node at (6,2)[circle,fill=white,inner sep=2pt]{$4$};
  \node at (6,4)[circle,fill,inner sep=5.5pt]{};
  \node at (6,4)[circle,fill=white,inner sep=2pt]{$4$};
  \node at (6,6)[circle,fill,inner sep=5.5pt]{};
  \node at (6,6)[circle,fill=white,inner sep=2pt]{$4$};
  \node at (6,8)[circle,fill,inner sep=5.5pt]{};
  \node at (6,8)[circle,fill=white,inner sep=2pt]{$4$};
  \node at (0,8)[circle,fill,inner sep=5.5pt]{};
  \node at (0,8)[circle,fill=white,inner sep=2pt]{$4$};
  \node at (2,8)[circle,fill,inner sep=5.5pt]{};
  \node at (2,8)[circle,fill=white,inner sep=2pt]{$4$};
  \node at (-2,10)[circle,fill,inner sep=5.5pt]{};
  \node at (-2,10)[circle,fill=white,inner sep=2pt]{$4$};
  \node at (0,10)[circle,fill,inner sep=5.5pt]{};
  \node at (0,10)[circle,fill=white,inner sep=2pt]{$4$};
  \node at (2,10)[circle,fill,inner sep=5.5pt]{};
  \node at (2,10)[circle,fill=white,inner sep=2pt]{$4$};
  \node at (4,10)[circle,fill,inner sep=5.5pt]{};
  \node at (4,10)[circle,fill=white,inner sep=2pt]{$4$};
  
\end{tikzpicture}
\end{center}
\caption{Number of robber's potential moves protected by chief} \label{queenprotecting}
\end{figure}

\begin{thm}
If all cops are on foot except one that is a chief and the robber is $m$-speedy, then 

\begin{equation*}
c(G) = 2\ceil*{\dfrac{m}{3}} + 1.
\end{equation*}
\end{thm}
\begin{proof}
Suppose all cops are on foot except a single cop who is a chief, with the robber being $m$-speedy.  By Lem. \ref{queenpositionlemma2}, we can assume that the chief protects the entire column in which the robber is located.  In the row the robber is located in, then, there are $m$ vertices on either side of the robber that the robber could potentially move to.  Since each cop on foot protects $3$ vertices in a given row, and if there are $f$ cops protecting the vertices on one side (the planar view of $G$) of the robber, we must have

\begin{equation*}
3f \geq m.
\end{equation*}

Thus, the minimal number of cops on foot on either side of the robber must be

\begin{equation*}
\ceil*{\dfrac{m}{3}},
\end{equation*}

\noindent placing a lower bound on the cop number of $G$ of

\begin{equation*}
c(G) \geq  2\ceil*{\dfrac{m}{3}} + 1.
\end{equation*}

To prove that this lower bound is actually the cop number of $G$, we must show that number of cops can guarantee capture of the $m$-speedy robber.  
Lemma \ref{queenpositionlemma2} allows us to initially place the chief in the same column as the robber, thus protecting this entire column.  Assume that each move by the robber is mirrored by the chief; that is, if the robber moves horizontally $k$ units, the chief then moves horizontally $k$ units and consequently protects the entire column the robber is now located in.  

Now, the robber is only allowed to move within the row it occupies.  For notation, let $f = \ceil*{\frac{m}{3}}$.  We claim that $2f$ cops on foot can capture the robber.

\begin{figure} 
\begin{center}
\begin{tikzpicture}[scale=.6]

 \foreach \x in {0,2,4,6,8,10,12,14,16}
   \foreach \y in {0,2}{
   \node at (\x,\y)[circle,fill,inner sep=2pt]{};
   }; 

   \node at (0,2)[circle,fill,inner sep=5pt]{};
   \node at (0,.8)[]{Cop};
   \draw (0,-.3)--(0,.4);
   \draw (0,1.4)--(0,2.3);
   \node at (6,2)[circle,fill,inner sep=5pt]{};
   \node at (6,.8)[]{Cop};
   \draw (6,-.3)--(6,.4);
   \draw (6,1.4)--(6,2.3);
   \node at (12,2)[circle,fill,inner sep=5pt]{};
   \node at (12,.8)[]{Cop};
   \draw (12,-.3)--(12,.4);
   \draw (12,1.4)--(12,2.3);
  
   \node at (17,2)[]{$\dots$};
   \node at (17,0)[]{$\dots$}; 
   
  \draw (-.5,0)--(16.5,0);
  \draw (-.5,2)--(16.5,2);
 
  \foreach \x in {2,4,8,10,14,16}{ 
  \draw (\x, -.3)--(\x, 2.3);
  };

\end{tikzpicture}
\end{center}
\caption{Initial first row placement of cops on foot} \label{placingcopsonfoot}
\end{figure}
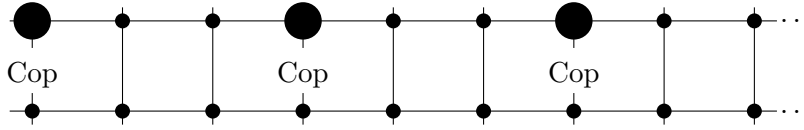

Initially place these $2f$ cops on foot in the first row of $G$, with one cop occupying the leftmost vertex and cops placed subsequently on every third vertex of the first row.  See Fig. \ref{placingcopsonfoot}.  Every time the cops move, have every cop on foot move one vertex down until they have reached the same row the robber is in.  Then, because $G$ is on the torus, we can consider the robber with $f$ cops located both on his right and on his left, as in Fig. \ref{robberbetweencops}

\begin{figure} 
\begin{center}
\begin{tikzpicture}[scale=.4]

 \foreach \x in {0,4,6,18,20,22,24} {
   \node at (\x,0)[circle,fill,inner sep=2pt]{};
   }; 
 
 \foreach \x in {2,8,16,22} {
  \node at (\x,0)[circle,fill,inner sep=5pt]{};
  }
  
  \node at (10,0)[]{$\dots$};
  \node at (14,0)[]{$\dots$};

  \draw (-.5,0)--(8.9,0);
  \draw (15.1,0)--(24.5,0); 
  
  \node at (2,-1.7)[]{Cop};
  \node at (8,-1.7)[]{Cop};
  \node at (16,-1.7)[]{Cop};
  \node at (22,-1.7)[]{Cop};

  \draw (11.1,0)--(12.9,0);
  \draw[fill=white] (12,0) circle (18pt);
  \node at (12,-1.7)[]{Robber};
\end{tikzpicture}
\end{center}
\caption{$f$ cops on either side of the robber} \label{robberbetweencops}
\end{figure}
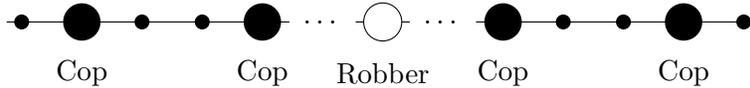

Every time the cops move, choose to have the cops on foot move one vertex towards the robber.  This preserves the string of $3f$ protected vertices on both the left and the right of the robber.  Moreover, because 

\begin{equation*}
f = \ceil*{\frac{m}{3}},
\end{equation*}

\noindent we have that 

\begin{equation*}
m \leq 3f.
\end{equation*}
 
This means that the robber will always remain between the $f$ cops on his left and the $f$ cops on his right (i.e., the robber cannot ``leapfrog" the entire string of $f$ cops on his left or right).  As the number of vertices between these two sets decreases with every move by the cops, the robber will eventually be captured.
\end{proof}

\section{Questions}
\label{futurequestions}

There are numerous immediate questions that arise from considering just these two scenarios.  Without varying the type of graph we are considering $(n \times n$ toroidal chess graphs), it is natural to look at the various other combinations of types of cops and robber.  What if the cops consist of a single chief and the rest as knights?  What if some of the cops are speedy?  In all of these cases, vary how the robber can move.  If the robber moves as a rook in chess, is a knight, or can move like a chief, how does the cop number change?  

In this paper we defined only movements mimicking the queen, pawns, knights, and rooks in chess.  If moves are defined in ways not mirroring chess pieces (such as a ``speedy knight" or a ``slow queen"), how does that change the game?  Additionally, every one of these situations can also be addressed from the lazy cop perspective \cite{Sullivan}, where only a single cop is allowed to move during each of the cops' turns.  What if cops are only allowed to move a certain number of times before the must ``take a mandatory break?"  If the robber is considered invisible (in both the adversarial and drunken situations), how does that impact the game \cite{Kehagias}?

The only graphs considered here are $n \times n$ chess graphs on the torus.  How is the game different, in all situations, when the chess graph is planar?  What if other graph-theoretic properties are introduced into the game (for example, weighting edges and restricting the ``total weight" allowed on any given move by the cops)?  What if the chess graphs, or portions of it, are directed?

It is not considered in this paper, but investigating the efficiency of a cop's strategy would be quite interesting \cite{Komarov}.  

The research done here was completed prior to one author (Hahn) taking an introduction-to-proofs course.  She was a full co-researcher in the process, despite having minimal exposure to post-calculus mathematics.  Similarly, the questions listed here are very accessible by undergraduates at any level and could serve as a door to experiencing mathematics research.

\end{document}